\pdfoutput=1
\documentclass[oneside, a4paper, 12pt]{amsart}
\usepackage{setspace}
\usepackage[inner=1in,outer=1in,tmargin=1in,bmargin=1in]{geometry}
\usepackage{mathtools}
\usepackage{amssymb}
\usepackage[utf8]{inputenc}
\usepackage[T1]{fontenc}
\usepackage[english]{babel}
\usepackage{parskip}
\usepackage{tensor}
\usepackage{hyperref}
\usepackage{amsmath}
\usepackage{afterpage}
\usepackage{tikz}
\usepackage{ae,aecompl}
\usepackage{graphicx}
\usepackage[capposition=bottom]{floatrow}
\usepackage{empheq}
\usepackage{amsthm}
\usepackage{lipsum} 
\usepackage{cite}
\usepackage{cases}
\usepackage{hyperref}
\usepackage{pgfplots}
\usepackage{enumitem}
\usepackage{pdfpages}
\usepackage{slashed}
\usepackage[colorinlistoftodos]{todonotes}

\allowdisplaybreaks

\hypersetup{
	colorlinks,
	linkcolor={red!80!black},
	citecolor={blue!80!black},
	urlcolor={blue!80!black},
	linktocpage=page
}

\usetikzlibrary{decorations.pathmorphing,calc}

\DeclareMathOperator{\Tr}{Tr}

\DeclarePairedDelimiter\floor{\lfloor}{\rfloor}

 \providecommand{\norm}[1]{\lVert#1\rVert}

\newcommand{\nn}{\nonumber}

\newcommand{\pt}{\partial}

\newcommand{\be}{\begin{equation}}
	\newcommand{\ee}{\end{equation}}
\newcommand{\ba}{\begin{eqnarray}}
	\newcommand{\ea}{\end{eqnarray}}
\newcommand{\baa}{\begin{array}}
	\newcommand{\eaa}{\end{array}}

\newcommand{\na}{\nabla}
\newcommand{\ijd}{_{ij}}
\newcommand{\iju}{^{ij}}

\newcommand{\nr}[1]{(\ref{#1})}  
\newcommand{\rmi}[1]{{\mbox{\scriptsize #1}}}  

\newcommand{\ria}{\rightarrow}

\newcommand{\pminus}[1]{^{-#1}}

\newcommand{\taux}{^{\tau(x,v)}}

\theoremstyle{definition}

\theoremstyle{definition}
\newtheorem{theorem}{Theorem}[section]

\theoremstyle{definition}
\newtheorem{lemma}{Lemma}[section]

\theoremstyle{definition}

\theoremstyle{definition}
\newtheorem{definition}{Definition}[section]

\theoremstyle{remark}
\newtheorem{remark}{Remark}
\numberwithin{equation}{section}

\newcommand{\LR}{\mathcal{L}}

\newcommand{\R}{\mathbb{R}}

\def\C{\mathbb C}

\newcommand{\Cinf}{C^\infty}
\newcommand{\Cinfo}{C^\infty_0}

\def\p{\partial}

\begin{document}

\title[The Light ray transform]{On the interplay between the light ray and the magnetic X-ray transforms}

\author[L. Oksanen]{Lauri Oksanen}
\address{Department of Mathematics and Statistics, University of Helsinki, Helsinki FI-00014, Finland}
\email{Lauri.oksanen@helsinki.fi}

\author[G.P. Paternain]{Gabriel P. Paternain}
\address{ Department of Mathematics, University of Washington, Seattle, WA 98195, USA}
\email {gpp24@uw.edu}

\author[M. Sarkkinen]{Miika Sarkkinen}
\address{Department of Mathematics and Statistics, University of Helsinki, Helsinki FI-00014, Finland}
\email{miika.sarkkinen@helsinki.fi}

\begin{abstract} We study the light ray transform acting on tensors on a stationary Lorentzian manifold. Our main result is injectivity up to the natural obstruction as long as the associated magnetic vector field satisfies a finite degree property with respect to the vertical Fourier decomposition on the unit tangent bundle. This is based on an explicit relationship between the geodesic vector field of the Lorentzian manifold and the magnetic vector field.

\end{abstract}

\maketitle


\section{Introduction}

This paper is concerned with the light ray transform of tensors on a stationary Lorentzian manifold. We start with a compact Riemannian $n$--manifold with boundary $(M, g)$ and a smooth 1-form $\omega$ on $M$. The pair $(g, \omega)$ gives rise to an $(n+1)$--dimensional Lorentzian manifold $\overline{M} = \mathbb{R} \times M$, where the Lorentzian metric is given by:
\begin{equation}\label{gstat2}
\overline{g} = -(dt + \omega)^2 + g.
\end{equation}
Here the 1-form $\omega$ and the Riemannian metric $g$ are independent of time $t$. 

Stationary spacetimes are usually defined by the condition that the metric possess a timelike Killing vector field \cite[Sec. 6.1]{Wald:1984rg}. 
If the timelike Killing vector has unit length, the spacetime is further called ultrastationary \cite{GHWW_09}. A metric of the above form clearly has a timelike Killing field $\pt_t$ and $\overline{g}(\pt_t, \pt_t) = -1$, so that our metric falls into the ultrastationary class. However, an ultrastationary metric is locally conformally related to a general stationary metric, and the object we are interested in -- the light ray transform -- is injective in the whole conformal class if it is injective for some metric in the class \cite[Cor. 2]{feizmohammadi2021light}. Therefore, in this paper we call a metric admitting gauge \nr{gstat2} stationary. 
Stationary Lorentzian metrics are of interest since one expects gravitating systems to settle down to a stationary spacetime at late times. For instance, all black hole solutions describing astrophysical black holes in their ground state are stationary.

It is well-known that null geodesics of stationary Lorentzian manifolds can be described in terms of magnetic geodesics on $(M, g)$, with the magnetic field given by $-d\omega$, see for instance \cite{Ger_07, GHWW_09}. In this paper, we take one step further and explore the relationship between transport equations with an eye towards understanding the injectivity properties of the light ray transform acting on tensors.

Let $G$ denote the infinitesimal generator of the magnetic flow on the unit sphere bundle $SM$. Recall that any function $f \in C^\infty(SM)$ can be expanded in terms of vertical spherical harmonics \cite{paternain2023geometric}.
We shall say that $f$ has degree $m$ if all Fourier modes for $k>m$ are zero. A function $f$ has finite degree if it has degree $m$ for some $m$.

Assume that $G$ is non-trapping, i.e., every magnetic geodesic hits $\partial M$ in finite time. We shall say that $G$ has the \emph{finite degree property} if given any $f\in C^{\infty}(SM)$ with finite degree, any solution $u\in C^{\infty}(SM)$ to 
\begin{equation}
G u = f,\;\;\;\;\;\;u|_{\partial SM}=0\label{eq:T}
\end{equation}
has finite degree. 
The main goal of this paper is to show that if $G$ has the finite degree property, then the light ray transform acting on tensors is injective, up to natural obstructions.

Recall that the null geodesics of $(\overline{M}, \overline{g})$ come with a unique parametrization up to an affine transformation. Without loss of generality we parameterize a null geodesic $\gamma$ such that its projection to $M$ has speed one with respect to $g$.

Given a symmetric $m$-tensor $\alpha \in C^\infty(S^m(T^*\overline{M}))$, we define the light ray transform of $\alpha$ as:
\begin{equation}
({\mathcal L}_m \alpha)(\gamma) = \int_{a}^{b} \alpha_{\gamma(s)}(\dot{\gamma}(s),\dots,\dot{\gamma}(s))\,ds. \label{LR-trafo}
\end{equation}
where $\gamma:[a,b]\to \overline{M}$ is a null geodesic connecting the boundary points $\gamma(a), \gamma(b) \in \partial M$. This transform has a natural kernel when $m\geq 1$. When $m=1$, the natural kernel is just exact forms with primitive vanishing on $\partial M$, while for $m\geq 2$ it consists of elements of the form
\begin{equation}\label{eq:kernel}
\overline{d}^{s}\beta+\xi\,\overline{g},
\end{equation}
where $\beta \in C^\infty(S^{m-1}(T^*\overline{M}))$ with $\beta |_{\partial M}=0$, $\overline{d}^{s}$ is the symmetrized covariant derivative and $\xi\,\overline{g}$ is the symmetric tensor product of $\overline{g}$ with $\xi \in C^\infty(S^{m-2}(T^*\overline{M}))$.

Previous work on the light ray transform includes, e.g., \cite{stefanov2017support,wang2017parametrices,ilmavirta2018x,lassas2018,feizmohammadi2021light,vasy2021light}. Our main theorem can be seen as a generalization of the injectivity result \cite[Thm. 2]{feizmohammadi2021light} on static Lorentzian manifolds.

\begin{theorem} Assume $G$ has the finite degree property and $\alpha \in C^\infty(S^m(T^*\overline{M}))$ has compact support in $\overline{M}$.
If $({\mathcal L}_m \alpha)(\gamma)=0$ for all null geodesics $\gamma$, then $\alpha$ must have the form given by \eqref{eq:kernel}.
\label{thm:main}
\end{theorem}

In all likelihood one can remove the assumption of compact support on $\alpha$ by replacing the finite degree assumption on $G$ with a corresponding assumption for the attenuated transport operator $G+(1-\omega)i\tau$, where $\tau\in \mathbb{R}$, see Remark \ref{remark:atte} below.

To make Theorem \ref{thm:main} effective, it remains for us to comment on conditions guaranteeing that $G$ (or the attenuated transport operator) satisfies the finite degree property. A full discussion of this is given in Subsection \ref{subs:Gfinite} below.

We will say that $G$ has the {\it degree m property} if given $f$ with degree $m$, a solution $u$ to \eqref{eq:T} has degree $m-1$. We will prove that $G$ has the finite degree property if and only if $G$ has the degree $m$ property for any $m\in \mathbb{Z}_{\geq 0}$ (see Lemma \ref{deg-m-prop}). Moreover, the degree $m$ property for $G$ is equivalent to the magnetic X-ray transform acting on pair of symmetric tensors of degree $m-1$ and $m$ being injective up to the natural obstruction (see Lemma \ref{I_m-inj}). Curiously, this natural obstruction has not been described before for $m\geq 3$; the case $m=2$ was considered in \cite{dairbekov2007}.

\subsection{Applications of light ray transform}

\subsubsection{Cosmological perturbation theory}

In observational cosmology, a large part of what we know about the early universe is based on measurements of the Cosmic Microwave Background (CMB), the `afterimage' of the Big Bang. Though possessing a highly uniform temperature distribution, there are small anisotropies in the measured CMB temperature that can be analyzed using cosmological perturbation theory. One contribution to CMB anisotropies comes from the integrated Sachs--Wolfe effect \cite{Sachs:1967er}, which is due to evolution of the gauge-invariant scalar perturbations (Bardeen potentials) along the CMB photon trajectory. Overall photon redshift is thus modified by an integral of the time derivatives of the scalar perturbations along the photon world line. This amounts to the light ray transform of a scalar function. 

More generally, one can show that the observed redshift perturbation accumulating along the photon trajectory is given by the light ray transform of a symmetric two-tensor that involves the metric perturbation and its first time derivative \cite[Thm. 4.2]{lassas2018}. Given that the small redshift anisotropies are known, one can then ask whether we can recover some information about the underlying metric perturbation, i.e. can the light ray transform be inverted.

\subsubsection{Lorentzian Calder\'on problem}

The light ray transform is related to 
the Lorentzian Calder\'on problem. There are various formulations of the problem, depending on the precise geometric setting. A brief formulation is as follows.
Let $(\overline{M},\overline{g})$ be a smooth Lorentzian manifold with timelike boundary, write $\Box_{\overline{g}}$ for the canonical wave operator on $\overline{M}$,
and consider the Cauchy data set 
    \begin{align}
\mathcal C_0 
= 
\{ (u|_{\p \overline{M}}, \p_\nu u|_{\p \overline{M}}) 
\mid
\text{$u \in H^2(\overline{M})$ solves $\Box_{\overline{g}} u = 0$ on $\overline{M}$}\}.
    \end{align}  
The problem is to find $\overline{g}$, up to an isometry fixing $\p \overline{M}$, given $\mathcal C_0$ and $\overline{M}$. 

It is typical to assume that there is smooth $\tau : \overline{M} \to \R$ with timelike differential $d\tau$. Then the wave operator is strictly hyperbolic with respect to the level surfaces of $\tau$ in the sense of \cite[Def. 23.2.3]{hormander2007}, and the natural initial-boundary value problem for the wave equation has a unique solution. 

To keep the discussion simple, let us suppose that $\tau$ is a proper map. Then $\overline{M}$ is exhausted by the compact preimages $\tau^{-1}([-T,T])$ for $T > 0$. Assuming further that all the null geodesics are non-trapped in the sense that their maximal domains of definition are compact intervals, we can apply \cite[Th. 1.3]{oksanen2024}, together with the boundary determination \cite[Th. 1.4]{oksanen2024}, to see that the Cauchy data set with the vanishing initial conditions
    \begin{align}
\mathcal C_1 
= 
\{ (u|_{\p \overline{M}}, \p_\nu u|_{\p \overline{M}}) 
\mid\ 
&\text{$u \in H^2(\overline{M})$ solves $\Box_{\overline{g}} u = 0$ on $\overline{M}$ and} \nn
\\
&\text{$u = 0$ when $\tau \le -T$ for some $T > 0$}\}
    \end{align}
determines the scattering relation 
    \begin{align}
\alpha(\gamma(a), \dot \gamma(a)) = (\gamma(b), \dot \gamma(b)),
    \end{align}
where $\gamma : [a,b] \to \overline{M}$ is a maximal null geodesic. 

A closely related earlier result \cite{stefanov2018} recovers the scattering relation without assuming the existence of $\tau$ from data that corresponds roughly to $\mathcal C_0$ with $\Box_{\overline{g}} u = 0$ replaced by $\Box_{\overline{g}} u \in C^\infty(\overline{M})$. The proof uses a parametrix construction and thus avoids actually solving the initial-boundary value problem for the wave equation. 

The scattering relation is uniquely associated, up to an overall factor, with the so-called smooth defining function. Letting $\tau \mapsto \overline{g}_\tau$ be a one-parameter family of Lorentzian metrics with $\overline{g}_0$ being some given metric, and $r_\tau$ the associated family of defining functions, the linearization $\pt_\tau r_\tau\rvert_{\tau=0}$ gives the light ray transform of the linear metric perturbation $\pt_\tau \overline{g}_\tau \rvert_{\tau=0}$ on $(\overline{M},\overline{g}_0)$ \cite[Thm. 2.1(c)]{stefanov2022lorentzian}. One is thus lead to consider the light ray transform of symmetric two-tensors.

In \cite{alexakis2022} the Lorentzian Calder\'on problem was solved in a fixed conformal class under certain geometric assumptions, the most important of which is a curvature bound in the sense of Andersson and Howard \cite{andersson1998}. In particular, the proof contains a step that recovers $\mathcal C_1$ from $\mathcal C_0$. 

Fixing the conformal class eliminates the diffeomorphism invariance in the problem. Moreover, when $\dim(\overline{M}) = 1+n$ and $n \ge 2$, the problem can be reduced to the problem to find $q \in C^\infty(\overline{M})$ given $(\overline{M},\overline{g})$ and 
    \begin{align}
\mathcal C_2 
= 
\{ (u|_{\p \overline{M}}, \p_\nu u|_{\p \overline{M}}) 
\mid\ 
&\text{$u \in H^2(\overline{M})$ solves $\Box_{\overline{g}} u + qu = 0$ on $\overline{M}$ and}
\\
&\text{$u = 0$ when $\tau \le -T$ for some $T > 0$}\}
    \end{align}  
The reduction is based on the rescaling
    \begin{align}
\Box_{c^{p-2} \overline{g}} (c^{-1} u) = c^{1-p} \Box_{\overline{g}} u + u c^{-p} \Box_{\overline{g}} c,
    \end{align}
where $c \in C^\infty(\overline{M})$ is positive and $p = 2(n+1)/(n-1)$, see e.g. \cite[Prop. 2.1]{liimatainen2022}.
Furthermore, it follows from \cite[Th. 1.3]{oksanen2024} that $\mathcal C_2$ determines the light ray transform of $q$. Here the assumptions are as above, that is, $\tau$ is a proper map with timelike differential and  all null geodesics are non-trapped.

Let us consider a related problem reducing to inversion of the light ray transform for one forms. Let $\nabla = d + a$ be the covariant derivative associated to a connection $a$ on the trivial line bundle $\overline{M} \times \C$. In other words, $a$ is a smooth one form. Consider the problem to recover $a$, up to the gauge $a \mapsto a + d\phi$, with $\phi$ a smooth function vanishing on $\p \overline{M}$, given $(\overline{M},\overline{g})$ and 
    \begin{align}
\mathcal C_3
= 
\{ (u|_{\p \overline{M}}, \nabla_\nu u|_{\p \overline{M}}) 
\mid\ 
&\text{$u \in H^2(\overline{M})$ solves $\nabla^* \nabla u = 0$ on $\overline{M}$ and}
\\
&\text{$u = 0$ when $\tau \le -T$ for some $T > 0$}\}.
    \end{align}
Here the adjoint is taken with respect to $\overline{g}$. 
Again, with the same assumptions on $\tau$ and the null geodesics, it follows from \cite[Th. 1.3]{oksanen2024} that $\mathcal C_3$ determines the parallel transport map with respect to $\nabla$, from $\gamma(a)$ to $\gamma(b)$,
along all maximal null geodesics $\gamma : [a,b] \to \overline{M}$, see also Example 1.9 there. As the parallel transport is simply the multiplication by $\exp(\int_\gamma a)$, we recover the light ray transform $\int_\gamma a$.

The equation $\nabla^* \nabla u = 0$ arises as a linearization of the equation satisfied by the pre-quantized Higgs field. A related inverse problem without linearization is treated in \cite{chen2022}. In this case, higher rank vector bundles $\overline{M} \times \C^k$ are of primary interest. With minor modifications, the proof in \cite{oksanen2024} gives the same reduction to the parallel transport for any rank $k$. However, when $k > 1$, it is an open question if a connection can be recovered, up to the gauge, from the parallel transport, even in the case of Minkowski geometry. 

The paper is organized as follows. In Section 2, we describe how null geodesics of a stationary Lorentzian metric are mapped, by a Hamiltonian reduction, to magnetic geodesics on the base Riemannian manifold. In Section 3, we show how the light ray transform is converted into a transport equation on the tangent bundle of the stationary manifold, and how this in turn is reduced by Fourier transform into a magnetic transport problem on the unit sphere bundle of the Riemannian manifold. Section 4 contains the proof of our main theorem. Finally, in Section 5 we provide a discussion on the finite degree property. In Appendices we give further details on some calculations and proofs omitted in the main text.

\subsection*{Acknowledgements}  LO and MS were supported by the European Research Council of the European Union, grant 101086697 (LoCal), and the Reseach Council of Finland, grants 347715, 353096 (Centre of Excellence of Inverse Modelling and Imaging) and 359182 (Flagship of Advanced Mathematics for Sensing Imaging and Modelling). Views and opinions expressed are those of the authors only and do not necessarily reflect those of the European Union or the other funding organizations. GPP was supported by NSF grant DMS-2347868 and is grateful to Claude Warnick for discussions related to \cite{GHWW_09}.

\section{From null geodesics to magnetic geodesics}

For the metric \nr{gstat2}, null geodesics on $\overline{M}$ can be reduced to magnetic geodesics on $M$. We will show this in the Hamiltonian formalism (cf. \cite{munoz2024scattering}). First, let us write the metric \nr{gstat2} in some coordinates $(t,x)$ where $x$ is a coordinate system on $M$
\begin{equation}
    \overline{g} = - (d t + \omega_i(x) d x^i )^2 + g\ijd (x) d x^i d x^j.
\end{equation}
Here we employed the notation where the symmetrized tensor product of any two symmetric tensors $S$ and $T$ is simply written as $S\, T$. Note then that the inverse metric for $\overline{g}$ is
\begin{equation}
    \overline{g}\pminus{1} = -(1-\omega_i \omega^i) \pt_t^2 - 2 \omega^i \pt_t \pt_i + g\iju \pt_i \pt_j ,
\end{equation}
where $(g\iju)$ is the matrix of the inverse metric $g\pminus{1}$ and $\omega^i = g\iju \omega_j$.
Any null tangent vector $\overline{v} = (v^t, v) \in T_p\overline{M}, v \in T_x M,$ then satisfies
\begin{equation}
    (v^t + \omega_x(v))^2 = \norm{v}_g^2 \ ,
\end{equation}
where $\norm{\cdot}_g$ is the norm induced by $g$. Similarly, the inverse metric $g\pminus{1}$ induces the norm $\norm{\cdot}_{g\pminus{1}}$ in the cotangent space. The metric \nr{gstat2} defines a Lagrangian $L: T\overline{M} \to \R$ by
\begin{equation}
    L(t,x,\dot{t},\dot{x}) = \frac{1}{2}\overline{g}((\dot{t},\dot{x}),(\dot{t},\dot{x})) = \frac{1}{2}\left( -(\dot{t}+\omega_x(\dot{x}))^2 + \norm{\dot{x}}_g^2\right).
\end{equation}
The conjugate momenta are given by
\begin{align}
    p_t &= \frac{\pt L}{\pt \dot{t}} = -(\dot{t} + \omega_x(\dot{x})), \\
    p &= \frac{\pt L}{\pt \dot{x}} = - (\dot{t} + \omega_x(\dot{x}))\omega_x + \dot{x}^\flat.
\end{align}
The Hamiltonian $H : T^*\overline{M} \to \R$ is then obtained from the Legendre transform of the Lagrangian
\begin{equation}
    H(t,x,p_t,p) = \frac{1}{2}\left( -p_t^2 + \norm{p - p_t \omega_x}_{g\pminus{1}}^2 \right).
\end{equation}
Null geodesics correspond to the energy level $\Sigma = H\pminus{1}(0)\setminus (\overline{M}\times\{0\})$ (the zero vector is spacelike). Time $t$ is a cyclic coordinate in the Hamiltonian so $p_t$ is conserved, which means that $p_t\rvert_\Sigma$ foliates $\Sigma$ into invariant hypersurfaces. We focus our attention on $p\pminus{1}_t(\{-1\})\subset \Sigma$ where we have 
\begin{equation}\label{dott}
    \dot{t} = 1 - \omega_x(\dot{x}).
\end{equation}
With this choice, we get $\norm{p + \omega_x}_{h\pminus{1}}^2 = 1.$ Now we can make the identification $\Sigma \cap \{p_t = -1\} \cong \R \times \Tilde{\Sigma}$ where $\Tilde{\Sigma}=\{ (x,p) \in T^*M : \norm{p + \omega_x}^2_{g\pminus{1}} = 1 \}$. The corresponding Hamiltonian on $p_t\pminus{1}(\{-1\})$,
\begin{equation}
    \Tilde{H}(x,p) = \frac{1}{2}\norm{p + \omega_x}^2_{g\pminus{1}},
\end{equation}
describes the motion of a charged particle in a magnetic field $-d\omega$ where $\omega$ is the potential one-form. Hamilton's equations read
\begin{equation}
    \dot{x}^i = \frac{\pt \Tilde{H}}{\pt p_i} = p^i, \quad \dot{p}_i = - \frac{\pt \Tilde{H}}{\pt x^i} = - \frac{1}{2}\pt_i g^{kl}(p_k + \omega_k) (p_l + \omega_l) - g^{kl}(p_k + \omega_k) \pt_i \omega_l.
\end{equation}
After a little bit of index gymnastics, these yield the well-known Lorentz force equation
\begin{equation}
    D_s \dot{x} = F(\dot{x}),
\end{equation}
where $D_s = \na_{\dot{x}}$ is the covariant derivative along $x(s)$ and the bundle map $F: TM \ria TM$ is given by $F^i_{\,\, j} = -(d\omega)^i_{\,\, j}$. We denote the transpose of $F$ by $F^*$. For later purposes, it is convenient to define the action of the operator $F^*$ on symmetric tensors of any rank $m=1,2,...$ by
 \begin{equation}
     F^*(\xi)_x(v_1,...,v_m) = \frac{1}{m}\left(
     \xi_x(F(v_1),v_2,...,v_m) + ... + \xi_x(v_1,...,F(v_m))
     \right), 
 \end{equation}
 where $\xi \in \Cinf(S^m(T^*M))$ and $v_1,...,v_m \in T_x M$, and setting $F^*(f) = f$ for any function $f \in \Cinf(M)$.
 Solutions to the Lorentz force equation are magnetic geodesics of $g$ determined by $F$. The pair $(g,\omega)$ defines a magnetic system on $M$; we also denote this by a triple $(M,g,\omega)$. 

Thus, null geodesics on $\overline{M}$ are mapped to magnetic geodesics on $M$ by the above reduction. On the other hand, given a magnetic geodesic $x$, there exists a null geodesic $\gamma$ that maps to $x$ under the canonical projection of $\overline{M}$ onto $M$, obtained by solving \nr{dott} for $t$, that is 
\begin{equation}\label{t(s)}
    t(s) = t_0 + s - \int_0^s \omega_{x(\sigma)}(\dot{x}(\sigma)) d\sigma.
\end{equation}
The geodesic $\gamma(s) = (t(s),x(s))$ is unique up to the choice of the initial time $t_0$. Notice that with the choice $p_t = -1$, the magnetic geodesics have unit normalization. 
Let $\varphi_s : SM \to SM$ be the magnetic flow of $(g,\omega)$ running with unit speed. The vector bundle of normalized null directions is the subset of $T\overline{M}$ given by 
\begin{equation}
    \{ (z,\overline{v}) \in T\overline{M}: z = (t,x) \in \overline{M}, \overline{v}=(v^0, v) \in T_z\overline{M} , \norm{v}_g = 1, v^0 = 1 - \omega_x(v)\}.
\end{equation}
Since the timelike component $v^0$ is fully determined by the magnetic potential and the unit spacelike direction, we denote points in the bundle by $(t,x,v)$ and the bundle itself by $\R \times SM$.
Then we define the normalized null geodesic flow $\phi_s : \R \times SM \to \R \times SM$ by
\begin{equation}
    \phi_s(t_0,x,v) = (t(s) , \varphi_s(x,v)),
\end{equation}
where $t(s)$ is given by \nr{t(s)}.
We denote by $X$ the generator of $\phi_s$ and by $G$ the generator of $\varphi_s$.

\section{Light ray transform and transport equation}

Given a point $(x,v) \in SM$ and a magnetic geodesic $\gamma_{x,v}$, the exit time of $\gamma_{x,v}$ is defined as the $\tau(x,v) \in [0,\infty]$ such that $\gamma_{x,v}:[0,\tau(x,v)] \to M$ cannot be extended in the positive direction as a smooth curve in $M$. We assume that $\varphi_s$ is non-trapping with exit time function $\tau: SM \to [0,\infty)$, i.e., all magnetic geodesics exit the manifold in finite time. Recall that the influx boundary $\pt_+ SM$ is defined as the set of inwards-pointing unit tangent vectors at the boundary $\pt M$ and, similarly, the outflux boundary $\pt_- SM$ is comprised of the outwards-pointing unit tangent vectors at the boundary, so that $\pt SM = \pt_- SM \cup \pt_+ SM$. The intersection $\pt_0 SM = \pt_+ SM \cap \pt_- SM = S(\pt M)$ is called the glancing region \cite[Ch. 3]{paternain2023geometric}.

Let us then define the light ray transform for smooth functions on $\R \times SM$.
\begin{definition}
    Given $f \in \Cinf(\R \times SM)$, we define the light ray transform of $f$ as
    \begin{equation}
        (\LR f)(t,x,v) = \int_0^{\tau(x,v)}f (\phi_s(t,x,v))ds, 
    \end{equation}
    where $(x,v) \in \pt_+ SM$.
\end{definition}
\noindent 
The transform is well-defined in the given class due to the non-trapping assumption, which ensures that the integration takes place over a finite interval. Below we will assume that in addition, $f \in \Cinfo(\R \times SM)$, i.e., $f$ is also compactly supported in time, which guarantees analyticity of Fourier transforms with respect to the time variable.
Earlier in \nr{LR-trafo} we defined the light ray transform $\LR_m$ of symmetric $m$-tensors along geodesics connecting two boundary points. We can re-express $\LR_m \alpha$ as a function on $\R \times SM$ by setting $(\LR_m \alpha)(t,x,v) = (\LR_m \alpha)(\gamma_{t,x,v})$ with $(x,v) \in \pt_+ SM$, where $\gamma_{t,x,v} : s \mapsto (\pi_{\overline{M}} \circ \phi_s)(t,x,v)$ and $\pi_{\overline{M}}: T\overline{M} \to \overline{M}$ is the canonical projection.

Recall that a function $f(x,v)$ in $\Cinf(SM)$ can be expanded at each point $(x,v)\in SM$ in terms of its spherical harmonic modes with respect to the direction $v$. A function is said to have degree $m$ if all the modes $k > m$ are vanishing. Similarly, a function $f(t,x,v)$ in $\Cinf(\R \times SM)$ has a series expansion in terms of spherical harmonics with respect to $v$ and we say that it has degree $m$ provided that all the modes $k>m$ vanish.

Given a symmetric tensor $\alpha \in \Cinf(S^m(T^*\overline{M}))$, we define a linear map 
$$T : \Cinf(S^m(T^*\overline{M})) \to \Cinf(\R \times SM)$$ by 
\begin{equation}
    T\alpha = (\pi_{\overline{M}}^*\alpha) (X,...,X).
\end{equation}
Notice that due to the fact that $T\overline{g}=0$, the above map is not an isomorphism. However, it only loses information about tensors that are already in the kernel of $\LR_m$ and therefore it makes the light ray transforms in the two classes equivalent. More concretely, since $\alpha$ can be written as
\begin{equation}
    \alpha = \sum_{j=0}^m \alpha_j dt^j,
\end{equation}
where $\alpha_j \in \Cinf(\R,\Cinf(S^{m-j}(T^*M)))$, the corresponding function $f = T\alpha$ on $\R \times SM$ becomes
\begin{align}
    f(t,x,v) &= (\pi_{\overline{M}}^*\alpha)(X,...,X)(t,x,v) \nn \\
    &= \alpha_{\pi_{\overline{M}}(t,x,v)}((\pi_{\overline{M}})_* X(t,x,v),...,(\pi_{\overline{M}})_*  X(t,x,v))  \nn \\
    &= \alpha_{\gamma_{t,x,v}(0)}(\dot{\gamma}_{t,x,v}(0),...,\dot{\gamma}_{t,x,v}(0)) \nn \\
    &= \sum_{j=0}^m \binom{m}{j} (1 - \omega_x(v))^j\alpha_j(t)_x(v,...,v).
\end{align}
Here we used the fact that the null geodesic $\gamma_{t,x,v} : s \mapsto (\pi_{\overline{M}} \circ \phi_s)(t,x,v)$ satisfies $\dot{\gamma}_{t,x,v}(0) = (1- \omega_x(v), v)$. Now it is easy to see that the light ray transforms defined in both classes of objects yield the same result:
\begin{align}
     (\LR f)(t,x,v) &= \int_0\taux (\pi_{\overline{M}}^*\alpha) (X(\phi_s(t,x,v)),...,X(\phi_s(t,x,v)))ds \nn \\
    &= \int_0\taux \alpha_{\gamma_{t,x,v}(s)}(\dot{\gamma}_{t,x,v}(s),...,\dot{\gamma}_{t,x,v}(s)) ds \nn \\
    &= (\LR_m \alpha) (t,x,v).
\end{align}
Put in another way, $\LR_m = \LR \circ T$.

Letting $(t,x,v)\in \R\times SM$, we define a function $u: \R \times SM \to \R$ by
\begin{equation}
    u(t,x,v) = \int_0\taux f(\phi_s(t,x,v)) ds.
\end{equation}
By the definition of exit time $\tau(x,v)$, we have $u\rvert_{\pt_- SM} \equiv 0$. Restricting on the influx boundary gives
\begin{equation}\label{u-restricted}
    u\rvert_{\R\times \pt_+ SM} = \LR f.
\end{equation}
On the other hand, definition of a geodesic vector field then yields
\begin{align}
    Xu(t,x,v) = \frac{d}{ds}\Big\rvert_{s=0} (u(\phi_s(t,x,v)) 
    = -f(t,x,v),
\end{align}
so that we get a transport problem $Xu = - f, u\rvert_{\R\times\pt_- SM} = 0$.
Combining this with \nr{u-restricted}, we see that $\LR f = 0$ if and only if the transport problem with overdetermined boundary conditions
\begin{equation}\label{transport0}
    Xu = - f, \quad u\rvert_{\R\times\pt SM} = 0
\end{equation}
has a solution. Furthermore, the solution is a smooth function, as shown in Lemma \ref{spacetime-u-lemma}.

The geodesic vector field $X$ can be expressed as
\begin{equation}\label{Xfield}
    X(t,x,v) = (1- \omega_x(v))\pt_t + G(x,v),
\end{equation}
which follows immediately from the definition of $\phi_s$. Using this form of $X$ and Fourier transforming \nr{transport0} with respect to $t$ (assuming that the Fourier transform exists) then reduces this into a transport problem on a time slice
\begin{equation}\label{transport1}
    G \hat{u}(\tau) + i\tau(1- \omega)\hat{u}(\tau) = - \hat{f}(\tau), \quad \hat{u}(\tau)\rvert_{\pt SM} = 0 \quad \forall \tau \in \R,
\end{equation}
where we denote $\hat{u}(\tau)= \hat{u}(\tau,\cdot)$ and $\hat{f}(\tau) = \hat{f}(\tau,\cdot)$.
This corresponds to an attenuated magnetic X-ray transform of $\hat{f}$ where the attenuation is given by $i\tau(1- \omega)$. 
Here we will only focus on the case $\tau = 0$, that is, we study the transport problem
\begin{equation}\label{G-eq}
        Gu = -f, \quad u\rvert_{\pt SM} = 0,
    \end{equation}
where $f \in \Cinf(SM)$ has degree $m$. 
The transport problem \nr{G-eq} corresponds to vanishing of the ordinary magnetic X-ray transform $I_m$ at $f$, defined as follows. Given a pair $[p,q]\in \Cinf(S^m(T^*M))\times \Cinf(S^{m-1}(T^*M))$ of symmetric tensors, we set
\begin{equation}
    I_m [p,q](x,v) = \int_0\taux \left( p_{\gamma_{x,v}(s)}(\dot{\gamma}_{x,v}(s),...,\dot{\gamma}_{x,v}(s)) + q_{\gamma_{x,v}(s)}(\dot{\gamma}_{x,v}(s),...,\dot{\gamma}_{x,v}(s)) \right) ds, 
\end{equation}
where $(x,v) \in \pt_+ SM$, $\gamma_{x,v} = s \mapsto (\pi_M \circ \varphi_s)(x,v)$ is a magnetic geodesic, and $\pi_M: TM \to M$ is the canonical projection. Then, $I_m[p,q] = 0$ if and only if the function $u: SM \to \R$ given by
\begin{equation}
    u(x,v) = \int_0\taux f(\varphi_s(x,v)) ds,
\end{equation}
where $f(x,v) = p_x(v,...,v) + q_x(v,...,v)$, is smooth and satisfies \nr{G-eq} (see Lemma \ref{u-lemma}). It is straightforward to check that $I_m[p,q]$ vanishes identically when $[p,q]$ is given by
 \begin{align}
     p &= \sum_{k=0}^{\floor{\frac{m-1}{2}}} \left( d^s \xi_{m-2k-1} g^k + (m-2k-2) F^*(\xi_{m-2k-2}) g^{k+1} \right) \\
     q &= \sum_{k=0}^{\floor{\frac{m-1}{2}}} \left( d^s \xi_{m-2k-2} g^k + (m-2k-1) F^*(\xi_{m-2k-1}) g^{k} \right),
 \end{align}
 where $\xi_i \in \Cinf(S^{i}(T^*M)), \xi_i\rvert_{\pt M} = 0$ for all $i = 0,1,...,m-1$. (Tensors of negative rank are identified with the zero map.) However, this can be simplified a lot by observing that for $n = 0,1,..., m - 2k$,
 \begin{align}
     F^*(\xi_{m-2k-n} g^k)(v,...,v) &= (\xi_{m-2k-n} g^k)(F(v),v,...,v) \nn \\
     &= \frac{1}{(m-n)!}\left(
     (m-n-1)! \xi_{m-2k-n}(F(v),v,...,v) g^k(v,...,v) \right. \nn \\
     &\left. \quad + \, ... + (m-n-1)! \xi_{m-2k-n}(v,...,v,F(v)) g^k(v,...,v)
     \right) \nn \\
     &= \frac{1}{m-n}\left( \xi_{m-2k-n} (F(v),v,...,v) + \, ...\right.\nn \\
     &\left. \quad  +\, \xi_{m-2k-n}(v,...,v,F(v)) \right) g^k(v,...,v) \nn \\
     &= \frac{m-2k -n}{m-n} F^*(\xi_{m-2k-n})(v,...,v) g^k(v,...,v),
 \end{align}
 from which we get the commutation relation
 \begin{equation}\label{F-g-comm}
     (m-2k-n)F^*(\xi_{m-2k-n}) g^k = (m-n) F^*(\xi_{m-2k-n} g^k).
 \end{equation}
 Then we define the symmetric tensors $\xi \in \Cinf(S^{m-1}(T^*M)), \eta \in \Cinf(S^{m-2}(T^*M))$ by
 \begin{align}
     \xi &= \sum_{k=0}^{\floor{\frac{m-1}{2}}} \xi_{m-2k-1}g^k \\
     \eta &=\sum_{k=0}^{\floor{\frac{m-1}{2}}} \xi_{m-2k-2}g^k,
 \end{align}
 so that $\xi\rvert_{\pt M} = \eta\rvert_{\pt M} = 0$.
 Using now identity \nr{F-g-comm} and metric compatibility of the connection, $p$ and $q$ can be simply written as
 \begin{align}
     p &= d^s \xi + (m-2) F^*(\eta) g \label{potential-p}\\
     q &= d^s \eta + (m-1) F^*(\xi), \label{potential-q}
 \end{align}
 We call a pair of this form a \emph{potential pair}.
 \begin{definition}
     We say that the magnetic X-ray transform $I_m$ is s-injective provided that $I_m[p,q] \equiv 0$ implies that $[p,q]$ is a potential pair.
 \end{definition}
 \begin{definition}
    We say that $G$ has the \emph{degree $m$ property} provided that the solution $u$ to the transport problem \nr{G-eq} has degree $m-1$ if $f$ has degree $m$.
\end{definition}
\noindent In Sec. \ref{finite-deg-section} we prove that this is actually equivalent to the finite degree property.
\begin{lemma}\label{deg-m-prop}
     $G$ has the finite degree property if and only if it has the degree $m$ property for all $m$.
 \end{lemma}
\noindent $G$ having the finite degree property or, equivalently, the degree $m$ property amounts to s-injectivity of $I_m$ as stated in the following lemma whose proof we also give in Sec. \ref{finite-deg-section}.
\begin{lemma}\label{I_m-inj}
     For $m\in \mathbb{Z}_{\geq 0}$, the magnetic X-ray transform $I_m$ is s-injective on $(M,g,\omega)$ if and only if $G$ has the degree $m$ property.
 \end{lemma}

\begin{lemma}\label{Xu-lemma}
    Let $f \in \Cinfo(\R \times SM)$ have degree $m$, and suppose $G$ has the degree $m$ property. Then the solution $u$ of the transport problem
    \begin{equation}\label{transport2}
        Xu = - f, \quad u\rvert_{\pt SM} = 0.
    \end{equation}
    has degree $m-1$.
\end{lemma}
\begin{proof}
    From \nr{transport2} we get that the Fourier transform $\hat{u}(\tau)\in \Cinf(SM)$ satisfies
    \begin{equation}\label{transport3}
        G \hat{u}(\tau) + i\tau(1- \omega)\hat{u}(\tau) = - \hat{f}(\tau), \quad \hat{u}(\tau)\rvert_{\pt SM} = 0
    \end{equation}
    for all $\tau$. We then claim that $\pt_\tau^k\hat{u}(0)$ has degree $m-1$ for each $k=0,1,2,...\,$. Let us prove this by induction. Taking $\tau=0$ in \nr{transport3}, we have
    \begin{equation}
        G\hat{u}(0) = - \hat{f}(0), \quad \hat{u}(0)\rvert_{\pt SM} = 0.
    \end{equation}
    Degree $m$ property then implies that $\hat{u}(0)$ has degree $m-1$.

    For the induction step, suppose the claim holds for $0,1,...,k-1$. Differentiating \nr{transport3} $k$ times and then setting $\tau = 0$, we get the transport problem
    \begin{equation}
        G \pt_\tau^{k}\hat{u}(0) = - \pt_\tau^{k}\hat{f}(0) - i k (1-\omega) \pt_\tau^{k-1} \hat{u}(0), \quad \pt_\tau^{k}\hat{u}(0)\rvert_{\pt SM} = 0.
    \end{equation}
    The right hand side has degree $m$ by the induction assumption; hence, by the degree $m$ property, $\pt_\tau^{k}\hat{u}(0)$ is a function of degree $m-1$.

    Now since $f$ is compactly supported in time, also $u$ has to be, and thus the Fourier transform of $\hat{u}$ is real analytic with respect to $\tau$. Hence, the Taylor series 
    \begin{equation}
        \hat{u}(\tau) = \sum_{k=0}^\infty \frac{\tau^k}{k!}\pt_\tau^k \hat{u}(0)
    \end{equation}
    converges. On the right hand side, each term in the series is a function of degree $m-1$. Since the projection onto each spherical harmonic mode is a bounded operator, it follows that the series converges for each mode separately. Thus, we conclude that the series converges to a function of degree $m-1$.
    Finally, by inverse Fourier transform back to $t$ we get that the solution $u$ to \nr{transport2} has degree $m-1$.
\end{proof}

\section{Main theorem}

With the general scheme developed above, we can prove our main theorem on injectivity of the light ray transform for symmetric tensors of any rank. Notice that to show injectivity, we only need to consider symmetric tensors of the form $\alpha = \alpha_1 dt + \alpha_2$ where $\alpha_1 \in \Cinfo(\R,\Cinf( S^{m-1}(T^*M))), \alpha_2 \in \Cinfo(\R,\Cinf(S^m(T^*M)))$. Indeed, using the equality $dt^2 = - \overline{g} - 2 \omega dt - \omega^2 + g$ one can show that tensors that are at least quadratic in $dt$ can be reduced to a sum of terms of the above form and terms proportional to the metric $\overline{g}$. The latter type of terms are automatically in the kernel of the light ray transform so those we may ignore.
\begin{theorem}
    Let $m\in \mathbb{Z}_{\geq 0}$. If $G$ has the degree $m$ property, then the light ray transform $\LR_m$ is injective (up to the natural kernel) for compactly supported symmetric tensors. 
\end{theorem}
\begin{proof}
    Let $f \in \Cinfo(\R \times SM)$ be a function corresponding to a symmetric tensor of rank 0, such that $\LR f \equiv 0$. Lemma \ref{Xu-lemma} then implies that $Xu = - f, u\rvert_{\pt SM} = 0$ has only an identically vanishing solution, i.e. $f = - X u = 0$. Hence, $\LR_0$ is injective.

    Let then $f \in \Cinfo(\R \times SM)$ correspond to a one-form $\alpha = \alpha_0 dt + \alpha_1$ such that $\LR f = \LR_1 \alpha \equiv 0$. 
    Now the solution to \nr{transport2} has degree 0, which means that, by \nr{Xfield}, we have
    \begin{equation}
        (1- \omega_x(v))\pt_t u(t,x) + (Gu)(t,x,v) = - f(t,x,v) = -\alpha_0(t,x)(1- \omega_x(v)) - \alpha_1(t)_x(v),
    \end{equation}
    which gives
    \begin{equation}
        (1- \omega_x(v))\pt_t u(t,x) + du(t)_x(v) = -\alpha_0(t,x)(1- \omega_x(v)) - \alpha_1(t)_{x}(v),
    \end{equation}
    where $d$ stands for the exterior derivative on $M$.
    By equating terms of the same parity in $v$ on both sides, we obtain
    \begin{equation}
        \alpha_0 = -\pt_t u, \quad \alpha_1 = -du.
    \end{equation}
    We thus have $\alpha = - \pt_t u dt - du = \overline{d}\psi$ for $\psi = -u$, where $\overline{d}$ is the exterior derivative on $\overline{M}$.

    For rank $m\geq 2$, consider $f \in \Cinfo(\R \times SM)$ corresponding to $\alpha = \alpha_1 dt + \alpha_2$ for $\alpha_1 \in \Cinfo(\R,\Cinf( S^{m-1}(T^*M))), \alpha_2 \in \Cinfo(\R,\Cinf(S^m(T^*M)))$, such that $\LR f = \LR_m \alpha \equiv 0$. Now Lemma \ref{Xu-lemma} yields a solution $u$ with degree $m-1$, and then Lemma \ref{G-lemma} implies that
    \begin{align}
        G u(t, x,v) 
        &= d^s \beta_{m-1}(t)(v,...,v) + (m-1) F^*(\beta_{m-1}(t))(v,...,v) \nn \\ 
        &\quad+ d^s \beta_{m-2}(t)(v,...,v) + (m-2) F^*(\beta_{m-2}(t))(v,...,v),
    \end{align}
    for some $\beta_{m-1} \in \Cinfo(\R, \Cinf(S^{m-1}(T^*M))), \beta_{m-2}\in \Cinfo(\R,\Cinf(S^{m-2}(T^*M)))$. Then we get
    \begin{align}
        &(1 - \omega_x(v))(\pt_t \beta_{m-1}(t)(v,...,v) + \pt_t\beta_{m-2}(t)(v,...,v)) \nn \\
        &+ d^s \beta_{m-1}(t)(v,...,v) + (m-1) F^*(\beta_{m-1}(t))(v,...,v) \nn \\ 
        &+ d^s \beta_{m-2}(t)(v,...,v) + (m-2) F^*(\beta_{m-2}(t))(v,...,v) \nn \\
        = & - \alpha_1(t)_x(v,...,v)(1 - \omega_x(v)) - \alpha_2(t)_x(v,...,v).
    \end{align}
    By parity and using the fact that $g(v,v) = 1$, we have
    \begin{align}
        \alpha_1 &= - d^s \beta_{m-2} - (m-1)F^*(\beta_{m-1}) - \pt_t \beta_{m-1} + \pt_t \beta_{m-2}\, \omega , \\ 
        \alpha_2 &= -d^s \beta_{m-1} - (m-1)\omega F^*(\beta_{m-1}) - \omega d^s\beta_{m-2} - (m-2)F^*(\beta_{m-2})g - \pt_t \beta_{m-2} (g - \omega^2),
    \end{align}
    so that
    \begin{align}
        \alpha &= \alpha_1 dt + \alpha_2 \nn \\
        &= -\pt_t \beta_{m-1} dt - (m-1)(dt+\omega)F^*(\beta_{m-1}) - d^s \beta_{m-1}  \nn \\
        &\quad - \pt_t \beta_{m-2} (g - \omega^2 - \omega dt) -  d^s \beta_{m-2}(dt + \omega) - (m-2) F^*(\beta_{m-2})g.
    \end{align}
    Using the formulas in Appendix \ref{christoffels}, one gets that
    \begin{equation}
        -\pt_t \beta_{m-1} dt - (m-1)(dt+\omega)F^*(\beta_{m-1}) - d^s \beta_{m-1} = - \overline{d}^s \beta_{m-1}.
    \end{equation}
    On the other hand, we can write $g - \omega^2 = \overline{g} + dt^2 + 2 \omega dt$ and then use again the formulas in the appendix to show that
    \begin{align}
        &- \pt_t \beta_{m-2} (g - \omega^2 - \omega dt) -  d^s \beta_{m-2}(dt + \omega) - (m-2) F^*(\beta_{m-2})g \nn \\ 
        =& - \overline{d}^s\beta_{m-2} (dt + \omega) - (\pt_t \beta_{m-2} + (m-2) F^*(\beta_{m-2})) \overline{g} \nn \\
        =& - \overline{d}^s ( \beta_{m-2} (dt + \omega)) - (\pt_t \beta_{m-2} + (m-2) F^*(\beta_{m-2})) \overline{g}.
    \end{align}
    It then follows that
    \begin{equation}
        \alpha = \overline{d}^s (-\beta_{m-1} - \beta_{m-2} (dt + \omega)) - (\pt_t \beta_{m-2} + (m-2) F^*(\beta_{m-2})) \overline{g} \in \ker \LR_m.
    \end{equation}
    Thus, we conclude that $\LR_m$ is injective (up to the natural kernel).
\end{proof}
\begin{remark}
    The main theorem together with Lemma \ref{I_m-inj} imply that that the light ray transform $\LR_m$ is injective for $m\in \mathbb{Z}_{\geq 0}$ provided that the corresponding magnetic X-ray transform $I_m$ is s-injective. 
\end{remark}
\begin{remark}\label{remark:atte}
    We expect that a similar procedure works for the attenuated magnetic X-ray transform. If this is the case, then we could formulate a version of the finite degree property for the transport equation \nr{transport1} and prove a counterpart of Lemma \ref{Xu-lemma} where we only need to assume that $f$ has a well-defined Fourier transform.
\end{remark}

\section{Finite degree property}\label{finite-deg-section}

\subsection{When does $G$ have the finite degree property?} \label{subs:Gfinite}

As promised in the Introduction, we now discuss conditions on $(M,g, \omega)$ that ensure that $G$ has the finite degree property.  All results assume that $\partial M$ is strictly magnetic convex, meaning that  $\Pi_{x}(v,v) > g_{x}(F_{x}(v),\nu(x))$ for all $x\in \partial M$ and $v\in T_{x}\partial M$.
Here $\Pi$ denotes the second fundamental form of $\partial M$ and $\nu$ the inward unit normal. 

Suppose first that $\dim M=2$. In this case \cite[Theorem 7.6]{Ains_13} asserts that $G$ has the finite degree property if $M$ is simply connected (i.e. a 2-disk) and $G$ is free of conjugate points (i.e. it is a {\it simple} magnetic system as defined in \cite{dairbekov2007}).

Suppose now that $\dim M\geq 3$ with the case $\dim M=3$ being the physically relevant dimension. As shown in \cite{dairbekov2007}, magnetic simplicity ensures that $G$ has the degree 0 and degree 1 property and under additional geometric assumptions (cf. \cite[Theorem 5.4]{dairbekov2007})  we also have the degree 2 property.
The strongest results are more recent and arise from the deployment of the Uhlmann-Vasy approach in \cite{uhlmann2016}. Here we need to assume that the system admits
a strictly convex function, meaning there is $f\in C^{\infty}(M)$ such that $G^{2}\pi^*f>0$, where $\pi:SM\to M$ is the canonical projection. In this instance, \cite[Theorem 1.3]{zhou2018} gives that $G$ has the degree 2 property. In fact the same approach should give the finite degree property (i.e. the degree $m$ property for any $m$); while this has not been carried out for magnetic systems, it has been implemented for the geodesic vector field \cite{dHUZ_19} and the approach is robust enough to carry over to $G$ (see for instance \cite{PUZ_19} for a related case).

Note that adding an attenuation of the form $i\tau(1-\omega)$ does not really have much impact on the results above about the finite degree property. Indeed since the attenuation
is purely imaginary the methods in \cite{Ains_13} work to cover this case when $\dim M=2$, and for $\dim M\geq 3$, the methods in \cite{paternain2019} deal with any possible matrix weight (not just scalar attenuations). 

\subsection{Proofs of Lemmas \ref{deg-m-prop} and \ref{I_m-inj}}
To prove Lemmas \ref{deg-m-prop} and \ref{I_m-inj} that characterize the finite degree property in different ways, we first give a couple of auxiliary results.
Using the notation and theorems in \cite[Sec. 6.6]{paternain2023geometric}, let $l_m$ be the canonical map between symmetric $m$-tensors and finite degree functions on $SM$
 \begin{equation}
     l_m : \Cinf(S^m(T^*M)) \to \Cinf(SM), \quad l_m h (x,v) = h_x(v,...,v),
 \end{equation}
 which induces a linear isomorphism
 \begin{equation}
     l_m : \Cinf(S^m(T^*M)) \to \bigoplus_{i=0}^{[m/2]}\Theta_{m-2j},
 \end{equation}
 where $\Theta_l\subset \Cinf(SM)$ is the set of eigenfunctions of degree $l$ of the vertical Laplacian. We may then prove the following lemma.
 \begin{lemma}\label{G-lemma}
     For any $\xi \in \Cinf(S^m(T^*M))$, $m\in \mathbb{Z}_{\geq 0}$, we have $G l_m \xi = l_{m+1} d^s \xi + m l_m F^*(\xi)$.
 \end{lemma}
 \begin{proof}
 Let $(x,v)\in SM$, $\xi \in \Cinf(S^m(T^*M))$, and let $\gamma : [0,T] \to N$ be a magnetic geodesic such that $(\gamma(0),\dot{\gamma}(0))=(x,v)$. Then
 \begin{align}
     l_{m+1}(d^s \xi)(x,v) &= (d^s \xi)_x(v,...,v) = (\na \xi)_x(v,...,v) \nn \\
     &= \frac{d}{dt}\Big\rvert_{t=0}\xi_{\gamma(t)}(\dot{\gamma}(t),...,\dot{\gamma}(t)) - \xi_{\gamma(0)}(\na_{\dot{\gamma}(0)}\dot{\gamma}(0),...,\dot{\gamma}(0)) \nn \\
     &\quad - ... - \xi_{\gamma(0)}(\dot{\gamma}(0),...,\na_{\dot{\gamma}(0)}\dot{\gamma}(0)) \nn \\
     &= \frac{d}{dt}\Big\rvert_{t=0}l_m\xi(\gamma(t),\dot{\gamma}(t)) - \xi_{\gamma(0)}(F(\dot{\gamma}(0)),...,\dot{\gamma}(0)) \nn \\
     &\quad - ... - \xi_{\gamma(0)}(\dot{\gamma}(0),...,F(\dot{\gamma}(0))) \nn \\
     &= \frac{d}{dt}\Big\rvert_{t=0}l_m\xi(\varphi_t(x,v)) - m F^*(\xi)_{\gamma(0)}(\dot{\gamma}(0),...,\dot{\gamma}(0)) \nn \\
     &= (G l_m \xi)(x,v) - m (l_m F^*(\xi))(x,v),
 \end{align}
 from which the claim follows.
 \end{proof}
 It is also straightforward to prove this lemma:
 \begin{lemma}\label{l-lemma}
     For all $m\in \mathbb{Z}_{\geq 0}$ and $k=0,...,\floor{m/2}$, we have $l_m (\xi_{m-2k} g^k) = l_{m-2k} \xi_{m-2k}$, where $\xi_{m-2k} \in \Cinf(S^{m-2k}(T^*M))$.
 \end{lemma}
 With these tools at hand, we can proceed to prove Lemma \ref{deg-m-prop}.
 \begin{proof}[Proof of Lemma \ref{deg-m-prop}]
     The `if' direction is trivial so we can focus on the `only if' part. Suppose $u$ satisfies $G u = f, u\rvert_{\pt SM} = 0$ for a degree $m$ function $f$. The finite degree property then implies that $\text{deg}(u) = n < \infty$. We can thus expand $u$ as a finite sum
     \begin{equation}
         u = \sum_{k=0}^n u_k,
     \end{equation}
     where each $u_k$ is a vertical spherical harmonic of degree $k$. We then get
     \begin{equation}
         Gu = \sum_{k=0}^n G u_k
     \end{equation}
     so we need to study how $G$ acts on spherical harmonics. For $k = 0$ this is trivial so we may assume $k\geq 1$. Due to the isomorphism between symmetric tensors and smooth functions on $SM$, for each $u_k$ there exists a unique $\xi \in \Cinf(S^k(T^*M))$ such that $u_k = l_k \xi$. By Lemma \ref{G-lemma}, we then have that
     \begin{equation}
         G u_k = G l_k \xi = l_{k+1} d^s \xi + k l_k F^*(\xi).
     \end{equation}
     
     Next we observe that $\xi$ is traceless. This follows from \cite[Lem. 2.4]{dairbekov2010} but for ease of reading we give here a proof. Indeed, if $\xi$ were not traceless, then we would have, by Lemma \ref{l-lemma}, that $l_k \xi = l_k ( \xi^\rmi{tf} + \frac{1}{n} \Tr_g(\xi) g) = l_k \xi^\rmi{tf} + \frac{1}{n}l_{k-2}\Tr_g(\xi)$, where `tf' denotes the tracefree part. Then the second term would be of degree $k-2$ and nonvanishing, which contradicts the fact that $\xi$ isomorphically corresponds to a spherical harmonic of degree $k$. Conversely, one can show that each traceless symmetric $k$--tensor is isomorphically mapped to a spherical harmonic of degree $k$: letting $\zeta \in \Cinf(S^k(T^*M))$ such that $\Tr_g \zeta = 0$, we write $l_k \zeta = \sum_{i=0}^k f_i$, which is equivalent to $\zeta = \sum_{i=0}^k l_k\pminus{1}f_i$. Now $l_k\pminus{1} f_i = \eta \in \Cinf(S^k(T^*M))$, but on the other hand, $f_i = l_i \zeta_i$ for some $\zeta_i \in \Cinf(S^i(T^*M))$ such that $\Tr_g(\zeta_i) = 0$, as we showed already. Thus, $l_k \eta = f_i = l_i \zeta_i$, which implies that $\eta = g^{(k-i)/2} \zeta_i$. Since $\zeta$ is traceless, this can hold only if $\zeta_i = 0$ whenever $k-i \geq 2$. Thus, $\zeta = l_k\pminus{1}f_k$ where $f_k$ is a spherical harmonic of degree $k$.

     Then, we decompose $d^s\xi$ to its trace and traceless parts
     \begin{equation}
         d^s \xi = (d^s \xi)^\rmi{tf} + \frac{1}{n}\Tr_g(d^s \xi) g.
     \end{equation}
     Notice that if $(d^s \xi)^\rmi{tf}$ were to vanish, then $\xi$ would be a traceless conformal Killing tensor field. However, the boundary condition $u\rvert_{\pt SM} = 0$ entails that $\xi = 0$ on the boundary, and \cite[Thm. 1.3]{dairbekov2010} implies that there are no traceless conformal Killing tensors that vanish on the boundary. Therefore, $(d^s \xi)^\rmi{tf}$ is nonzero. Further, $l_{k+1}(d^s \xi)^\rmi{tf}$ has degree $k+1$ and it cannot have any degree smaller than that due to the isomorphism between traceless symmetric tensors and spherical harmonics. This then implies that $Gu_k$ has degree $k+1$. Applying this to $k=n$, we immediately get that $m = \text{deg}(f) = \text{deg}(G u_n) = n + 1$, which gives that $u$ has degree $m-1$.
 \end{proof}
 
 \begin{definition}
      Given a function $f \in \Cinf(SM)$, its magnetic X-ray transform is defined as
      \begin{equation}
          (If)(x,v) = \int_0\taux f(\varphi_s(x,v))ds,
      \end{equation}
      where $(x,v) \in \pt_+ SM$.
 \end{definition}
 \noindent Notice that the magnetic X-ray transform $I_m$ of a symmetric tensor pair $[p,q]$ could then also be defined as $I_m[p,q] = I(l_m p + l_{m-1}q)$.
 
 Now we can prove Lemma \ref{I_m-inj}.
 \begin{proof}[Proof of Lemma \ref{I_m-inj}]
     `only if' Let $u$ and $f$ be such that the conditions of the lemma are satisfied. Then, since $u\rvert_{\pt SM} = 0$, by uniqueness for the transport equation we have $If = u\rvert_{\pt_+ SM} = 0$. We can split $f$ into even and odd pieces as $f = f^\rmi{even} + f^\rmi{odd}$. Suppose $m$ is even (the odd case proceeds analogously). Then by the isomorphism between symmetric tensors and functions on the unit sphere bundle, there are unique $p \in \Cinf(S^m(T^*M)), q \in \Cinf(S^{m-1}(T^*M))$ such that $f^\rmi{even} = l_m p$ and $f^\rmi{odd} = l_{m-1} q$. Now we have $I_m [p,q] = I(l_m p)  + I (l_{m-1}q) = I f^\rmi{even} + I f^\rmi{odd} = I f = 0$. S-injectivity of $I_m$ then entails that $[p,q]$ is a potential pair, i.e. a pair given by \nr{potential-p} and \nr{potential-q}.
     Now Lemma \ref{G-lemma} implies that
     \begin{align}
         G (l_{m-1}\xi + l_{m-2}\eta) &= l_m d^s \xi + (m-1) l_{m-1}F^*(\xi) + l_{m-1} d^s\eta + (m-2)l_{m-2}F^*(\eta) \nn \\
         &= l_m (d^s \xi + (m-2) F^*(\eta)g) + l_{m-1} (d^s\eta + (m-1) F^*(\xi)) \nn \\
         &= l_m p + l_{m-1} q \nn \\
         &= f^\rmi{even} + f^\rmi{odd} \nn \\
         &= f,
     \end{align}
     where on the second line we used Lemma \ref{l-lemma}.
     Thus, we get
     \begin{equation}
         G(u + l_{m-1}\xi + l_{m-2}\eta) = -f + f = 0.
     \end{equation}
     Since $u$ and $l_{m-1}\xi + l_{m-2}\eta$ both vanish on $\pt SM$, by uniqueness for the transport equation we have $u = - l_{m-1}\xi - l_{m-2}\eta$, which means that $u$ has degree $m-1$.

     `if' Suppose that $I_m[p,q] \equiv 0$. Then define $u: SM \to \R$ by
     \begin{equation}
         u(x,v) = - \int_0\taux (l_m p + l_{m-1}q)(\varphi_s(x,v))ds.
     \end{equation}
     Since $u\rvert_{\pt_+ SM} = I_m[p,q] = 0$, it follows from Lemma \ref{u-lemma} that $u$ is smooth and it is the unique solution to
     \begin{equation}\label{p-q-eq}
         G u = l_m p + l_{m-1} q, \quad u\rvert_{\pt SM} = 0.
     \end{equation} 
     The right hand side of \nr{p-q-eq} has degree $m$ so by the degree $m$ property of $G$, $u$ has degree $m-1$. Thus, for some $\xi \in \Cinf(S^{m-1}(T^*M)), \eta \in \Cinf(S^{m-2}(T^*M))$, we have
     \begin{equation}
         u = l_{m-1} \xi + l_{m-2} \eta.
     \end{equation}
     Now by Lemma \ref{G-lemma}
     \begin{align}
         G u &= G l_{m-1} \xi + G l_{m-2} \eta \nn \\
         &= l_m d^s \xi + (m-1) l_{m-1} F^*(\xi) + l_{m-1} d^s \eta + (m-2) l_{m-2} F^*(\eta) \nn \\
         &= l_m (d^s \xi + (m-2) F^*(\eta) g) + l_{m-1}(d^s \eta + (m-1) F^*(\xi)).
     \end{align}
     Matching this with the functions of the same parity in \nr{p-q-eq}, and using the inverse maps $l_{m}\pminus{1}, l_{m-1}\pminus{1}$, we obtain
     \begin{align}
         p &= d^s \xi + (m-2) F^*(\eta) g \\
         q &= d^s \eta + (m-1) F^*(\xi).
     \end{align}
     Thus, $[p,q]$ is a potential pair.
 \end{proof}

\newpage

\appendix

\section{Properties of the stationary Lorentzian metric
}\label{christoffels}

First we write down the Christoffel symbols for the metric $\overline{g}$ in \nr{gstat2}
\begin{align}
    &\overline{\Gamma}^t_{tt} = 0, \quad \overline{\Gamma}^i_{tt} = 0, \quad \overline{\Gamma}^t_{ti} = \frac{1}{2}(d\omega)\ijd \omega^j, \quad \overline{\Gamma}^i_{tj} = \frac{1}{2}(d\omega)^i_{\,\, j}, \quad  \nn \\
    &\overline{\Gamma}^i_{jk} = \Gamma^i_{jk} + \frac{1}{2}g^{il}(\pt_l (\omega_j \omega_k) - \omega_j \pt_k \omega_l - \omega_k \pt_j \omega_l) = \Gamma^i_{jk} + \frac{1}{2}( \omega_j (d\omega)^i_{\,\, k} + \omega_k (d\omega)^i_{\,\, j} ) \nn \\
    &\overline{\Gamma}^t\ijd = (d^s\omega)\ijd + \frac{1}{2}(\omega_i (d\omega)_{jk} + \omega_j (d\omega)_{ik})\omega^k
\end{align}
where the spatial indices are raised with $g\pminus{1}$, and $\Gamma^i_{jk}$ are the Christoffel symbols for the metric $g$. 
Using the above Christoffel symbols, one can show by a direct (but somewhat tedious) computation in local coordinates that 
\begin{equation}
     \overline{d}^s U = \pt_t U dt  + d^s U + m (dt + \omega) F^*(U),
 \end{equation}
 for any $U\in \Cinf(\R,\Cinf(S^m(T^*M)))$ with $m =0,1,2,...$. As a special case, we get
 \begin{equation}\label{domega}
     \overline{d}^s \omega = d^s\omega + (dt + \omega)F^*(\omega) .
 \end{equation}
 Furthermore, by another computation in local coordinates we have
 \begin{equation}\label{ddt}
     \overline{d}^s dt = - d^s \omega - (dt + \omega)F^*(\omega).
 \end{equation}
 Together \nr{domega} and \nr{ddt} immediately imply
 \begin{equation}
     \overline{d}^s (dt + \omega) = 0.
 \end{equation}
 Another way to see the last equality is by observing that $(dt + \omega)^\sharp = - \pt_ t$, and that $\pt_t$ is manifestly a Killing vector field of $\overline{g}$.

 \section{Regularity results for transport equations}
 
 Recall from Sec. \ref{subs:Gfinite} that $(M,g,\omega)$ forms a simple magnetic system if $\pt M$ is strictly magnetic convex, $M$ is simply connected, and the magnetic flow does not have any conjugate points.
 
 \begin{lemma}\label{u-lemma}
     Let $(M,g,\omega)$ be a simple magnetic system. If $f \in \Cinf(SM)$ is such that $I f = 0$, then the function $u: SM \to \R$ given by
     \begin{equation}
         u(x,v) = \int_0\taux f(\varphi_s(x,v)) ds,
     \end{equation}
     is smooth and solves the transport problem \nr{G-eq}.
 \end{lemma}
 \begin{proof}
     The proof is similar to that of \cite[Prop. 4.3]{Ains_13}, but for ease of reading, we state the proof here explicitly. Let us first show that $u$ satisfies \nr{G-eq}. $Gu = -f$ follows immediately from the definition of the magnetic vector field and properties of the exit time. Also, $u\rvert_{\pt_- SM}$ is given by an integral over a zero-measure set and therefore vanishes. On the other hand, $u\rvert_{\pt_+ SM} = I_m[p,q] = 0$ by assumption.
     
     Next, let $\alpha : \pt_+ SM \to \pt_- SM$ be the scattering relation defined by $\alpha(x,v) = (\gamma_{x,v}(\tau(x,v)),\dot{\gamma_{x,v}}(\tau(x,v)))$, where $\gamma_{x,v}$ is the unique magnetic geodesic determined by the initial data $(x,v)$. Given $w \in \Cinf(\pt_+ SM)$, the function $w^\sharp : SM \to \R$ defined by $w^\sharp(x,v) = (w \circ \alpha\pminus{1})(\varphi_{\tau(x,v)}(x,v))$ is the unique solution to
     \begin{equation}
         G w^\sharp = 0, \quad w^\sharp\rvert_{\pt_+ SM} = w.
     \end{equation}
     Then we define the extension operator $A: w \mapsto Aw$ as follows
     \begin{equation}
         (Aw)(x,v) = \begin{cases}
             w(x,v), & (x,v) \in \pt_+ SM, \\
             (w \circ \alpha\pminus{1})(x,v), & (x,v) \in \pt_- SM.
         \end{cases}
     \end{equation}
     It is immediate that $Aw = w^\sharp\rvert_{\pt SM}$. Now, Lemma 7.6 of \cite{dairbekov2007} implies that $w^\sharp \in \Cinf(SM)$ if and only if $Aw \in \Cinf(\pt SM)$.
     
     Assume that $M$ is isometrically embedded into a closed manifold $N$, then extend $g$ and $\omega$ smoothly to a Riemannian metric and one-form on $N$ and $f$ to a smooth function on the unit sphere bundle $SN$. Let then $U$ be an open neighborhood of $M$ in $N$ with smooth boundary, such that $(\overline{U},g,\omega)$ is a simple magnetic system (such $U$ exists since a simple magnetic system is defined by an open condition). 

     Let then $r$ be the unique solution to the transport problem on $S\overline{U}$
     \begin{equation}
         G r = - f, \quad r\rvert_{\pt_- S\overline{U}} = 0.
     \end{equation}
     Then $\Tilde{r} = r\rvert_{SM}$ is a smooth function that satisfies $G r = - f$. Observe now that the difference $\Tilde{r} - u$ satisfies $G(\Tilde{r}-u) = 0$, and the function $w = (\Tilde{r} - u)\rvert_{\pt_+ SM}$ is smooth due to the fact that $\Tilde{r}$ is smooth and $u\rvert_{\pt_+ SM} = 0$. Then $w^\sharp = \Tilde{r} - u$, which is smooth if and only if $u$ is smooth. But now $Aw = w^\sharp\rvert_{\pt SM} = (\Tilde{r} - u)\rvert_{\pt SM} = \Tilde{r}\rvert_{\pt SM}$ and therefore $Aw$ is smooth. By what we found earlier, this implies that $w^\sharp$ is smooth. Hence, $u$ is also smooth. 
 \end{proof}

 \begin{lemma}\label{spacetime-u-lemma}
    Let $(M,g,\omega)$ be a simple magnetic system. If $f \in \Cinf(\R \times SM)$ satisfies $\LR f = 0$, then the solution $u : \R \times SM \to \R$ of \nr{transport0}, given by
    \begin{equation}
        u(t,x,v) = \int_0\taux f(\phi_s(t,x,v)) ds,
    \end{equation}
    is smooth.
\end{lemma}
To prove the above lemma, we first define the \emph{enter time} of a magnetic geodesic: given a maximal magnetic geodesic $\gamma_{x,v}: [-\sigma(x,v),\tau(x,v)] \to M$ with an initial point $(x,v) \in SM$, we call $\sigma(x,v)$ its enter time. Notice that $\sigma(x,v) \neq \tau(x,-v)$ for magnetic geodesics in general. Then we need to know how the enter and exit times behave near a point in $\pt_0 SM$ (cf. \cite[Lemma 3.29]{paternain2023geometric}).
\begin{lemma}\label{tau-lemma}
    Let $(M,g,\omega)$ be a simple magnetic system, and let $(x_0,v_0) \in \pt_0 SM$. Provided that $M$ is embedded in a closed manifold $N$ of the same dimension, it holds near $(x_0,v_0)$ in $SM$ that
    \begin{align}
            \tau(x,v) &= Q(\sqrt{a(x,v)},x,v) \\
            -\sigma(x,v) &= Q(-\sqrt{a(x,v)},x,v),
    \end{align}
    where $Q$ is smooth in a neighborhood of $(0,x_0,v_0)$ in $\R \times SN$, $a$ is smooth in a neighborhood of $(x_0,v_0)$ in $SN$, and $a \geq 0$ in $SM$.
\end{lemma}
\begin{proof}
    Let $(N,g)$ be a closed extension of $(M,g)$, and let $\rho \in \Cinf(N)$ be a boundary defining function for $M$, such that $\na \rho$ is the inside pointing unit normal on $\pt M$ (see \cite[Lemma 3.1.10]{paternain2023geometric}). Then define the function $h(s,x,v) = \rho(\gamma_{x,v}(s))$, which is smooth in a neighborhood of $(0,x_0,v_0)$. Now, $h(0,x_0,v_0) = 0$, and
    \begin{equation}
        \pt_s h (0,x_0,v_0) = d\rho (\dot{\gamma}_{x_0,v_0}(0)) = g(\na \rho , v_0) = 0
    \end{equation}
    since $v_0 \in T_{x_0} \pt M$. Further,
    \begin{align}
        \pt_s^2 h (0,x_0,v_0) &= g(\na_{\dot{\gamma}_{x_0,v_0}}\na\rho, v_0) + g(\na\rho, \na_{\dot{\gamma}_{x_0,v_0}(0)} \dot{\gamma}_{x_0,v_0} (0)) \nn \\
        &= g (\na_{v_0} \na \rho, v_0) + g(\na\rho, F(v_0)) \nn \\
        &= - \Pi(v_0,v_0) + g(\nu(x_0), F(v_0)) < 0,
    \end{align}
    by strict magnetic convexity. Then, notice that $h(\tau(x,v),x,v) = h(-\sigma(x,v),x,v) = 0$ for any $(x,v) \in SM$; in particular, this holds near $(0,x_0,v_0)$. Lemma 3.2.10 in \cite{paternain2023geometric} now implies that $\tau(x,v) = Q (\pm \sqrt{a(x,v)},x,v)$ and $-\sigma(x,v) = Q(\mp \sqrt{a(x,v)},x,v)$, where the function $Q$ is smooth near $(0,x_0,v_0)$ in $\R \times SN$ and has $Q(\sqrt{a(x,v)},x,v) \geq Q(-\sqrt{a(x,v)},x,v)$ when $(x,v) \in SM$, and the function $a$ is smooth near $(x_0,v_0)$ in $SN$ and satisfies $a(x,v) \geq 0$ when $(x,v) \in SM$. Then, from the fact that $-\sigma(x,v) \leq \tau(x,v)$ we finally get our claim.
\end{proof}

\begin{proof}[Proof of Lemma \ref{spacetime-u-lemma}]
    Embed $(M,g)$ isometrically in a closed manifold $(N,g)$ with the same dimension, and extend $\omega$ and $f$ smoothly to $N$ and $SN$, respectively. Define a function $F: \R^2 \times SN \to \R$ by
    \begin{equation}
        F(r,t,x,v) = \int_0^r f(\phi_s(t,x,v))ds.
    \end{equation}
    Observe that $F$ is naturally smooth everywhere, and $u(t,x,v) = F(\tau(x,v),t,x,v)$. Let $S: \R\times\pt_+ SM \to \R\times\pt_- SM$ be the scattering relation of null geodesics in $\overline{M}$. Now, since $\LR f = 0$ (in the submanifold $\R \times SM$), we have, for any $(t,x,v) \in \R \times SM$,
    \begin{align}
        0 &= (\LR f \circ S\pminus{1})(\phi_{\tau(x,v)}(t,x,v)) = \int_{-\sigma(x,v)}\taux f(\phi_s(t,x,v))ds \nn \\ 
        &= \int_{-\sigma(x,v)}^0 f(\phi_s(t,x,v))ds + \int_0\taux f(\phi_s(t,x,v))ds.
    \end{align}
    Thus,
    \begin{align}
        F(\tau(x,v),t,x,v) &= \int_0\taux f(\phi_s(t,x,v))ds \nn \\
        &= - \int_{-\sigma(x,v)}^0 f(\phi_s(t,x,v))ds \nn \\ 
        &= \int^{-\sigma(x,v)}_0 f(\phi_s(t,x,v))ds \nn \\
        &= F(-\sigma(x,v),t,x,v),
    \end{align}
    so that we may write
    \begin{equation}
        F(\tau(x,v),t,x,v) = \frac{1}{2}(F(\tau(x,v),t,x,v) + F(-\sigma(x,v),t,x,v)).
    \end{equation}
    With a simple argument involving the implicit function theorem, one can see that the exit time $\tau$ is smooth everywhere in $SM \setminus \pt_0 SM$ if $\pt M$ is strictly magnetic convex. Therefore, $u$ is obviously smooth in $\R \times (SM \setminus \pt_0 SM)$. To show that $u$ is also smooth in the rest of $\R \times SM$, let $(t_0,x_0,v_0) \in \R \times \pt_0 SM$. By Lemma \ref{tau-lemma}, we have
    \begin{equation}
        F(\tau(x,v),t,x,v) = \frac{1}{2}(F(Q(\sqrt{a(x,v)},x,v),t,x,v) + F(Q(-\sqrt{a(x,v)},x,v),t,x,v)).
    \end{equation}
    near $(t_0,x_0,v_0)$ in $\R \times SM$.
    Now as in the proof of \cite[Thm. 5.1.1]{paternain2023geometric}, we deduce the existence of a function $H$ that is smooth near $(0,t_0,x_0,v_0)$ in $\R \times SM$ and satisfies
    \begin{equation}
        H(r^2, t,x,v) =  \frac{1}{2}(F(Q(r,x,v),t,x,v) + F(Q(-r,x,v),t,x,v)).
    \end{equation}
    Hence,
    \begin{equation}
        u(t,x,v) = H(a(x,v),t,x,v)
    \end{equation}
    near the point $(t_0,x_0,v_0)$ in $\R \times SM$, so that $u$ is smooth in a neighborhood of $(t_0,x_0,v_0)$ in $\R \times SM$.
\end{proof}

\bibliographystyle{alpha}
\bibliography{main}

\end{document}